\documentclass[reqno]{amsart}
\usepackage{amsmath, amsthm, amssymb, amstext}

\usepackage{hyperref,xcolor}
\hypersetup{
 pdfborder={0 0 0},
 colorlinks,
}

\usepackage{enumitem}
\newtheorem{theorem}{Theorem}
\newtheorem{remark}[theorem]{Remark}

\newtheorem{proposition}[theorem]{Proposition}
\newtheorem{corollary}[theorem]{Corollary}

\DeclareMathOperator*{\divergenz}{div}              %
\DeclareMathOperator*{\ints}{int}         %
\DeclareMathOperator*{\essinf}{ess \,inf}         %
\DeclareMathOperator*{\ww}{w}         %
\DeclareMathOperator*{\Ss}{S}         %

\newcommand{\N}{\mathbb{N}}

\newcommand{\R}{\mathbb{R}}

\newcommand{\Lp}[1]{L^{#1}(\Omega)}

\newcommand{\Wp}[1]{W^{1,#1}(\Omega)}

\newcommand{\Wpzero}[1]{W^{1,#1}_0(\Omega)}

\newcommand{\lan}{\langle}
\newcommand{\ran}{\rangle}
\newcommand{\eps}{\varepsilon}
\newcommand{\ph}{\varphi}

\newcommand{\into}{\int_{\Omega}}
\newcommand{\weak}{\overset{\ww}{\to}}

\newcommand{\Linf}{L^{\infty}(\Omega)}
\newcommand{\close}{\overline{\Omega}}
\newcommand{\interior}{\ints \left(C^1_0(\overline{\Omega})_+\right)}

\newcommand{\cprime}{$'$}

\renewcommand{\l}{\left}
\renewcommand{\r}{\right}

\newcommand*\diff{\mathop{}\!\mathrm{d}}
\newcommand{\W}{W^{1,p(\cdot)}_{0}(\Omega)}
\newcommand{\LP}{L^{p(\cdot)}(\Omega)}

\numberwithin{theorem}{section}
\numberwithin{equation}{section}

\title[Positive solutions for singular anisotropic $(p,q)$-equations]{Positive solutions for singular anisotropic $(p,q)$-equations}

\author[N.\,S.\,Papageorgiou]{Nikolaos S.\,Papageorgiou}
\address[N.\,S.\,Papageorgiou]{National Technical University, Department of Mathematics, Zografou Campus, Athens 15780, Greece}
\email{npapg@math.ntua.gr}

\author[P.\,Winkert]{Patrick Winkert}
\address[P.\,Winkert]{Technische Universit\"{a}t Berlin, Institut f\"{u}r Mathematik, Stra\ss e des 17.\,Juni 136, 10623 Berlin, Germany}
\email{winkert@math.tu-berlin.de}


\begin{document}

\begin{abstract}
	In this paper we consider a Dirichlet problem driven by an aniso\-tropic $(p,q)$-differential operator and a parametric reaction having the competing effects of a singular term and of a superlinear perturbation. We prove a bifurcation-type theorem describing the changes in the set of positive solutions as the parameter moves. Moreover, we prove the existence of a minimal positive solution and determine the monotonicity and continuity properties of the minimal solution map.
\end{abstract}

\maketitle

\section{Introduction}

Let $\Omega \subseteq \R^N$ be a bounded domain with a $C^2$-boundary $\partial \Omega$. In this paper, we deal with the following parametric anisotropic singular $(p,q)$-equation
\begin{align}\tag{P$_\lambda$}\label{problem}
  \begin{split}
    &-\Delta_{p(\cdot)}u-\Delta_{q(\cdot)} u = \lambda \l[ u^{-\eta(x)}+f(x,u)\r]\quad \text{in } \Omega, \\
    &u\big|_{\partial \Omega}=0, \quad u>0, \quad \lambda>0.
   \end{split}
\end{align}

Given $r \in C(\close)$ we define
\begin{align*}
	r_-=\min_{x\in \close}r(x) \quad\text{and}\quad r_+=\max_{x\in\close} r(x)
\end{align*}
and introduce the set 
\begin{align*}
E_1=\l \{r \in C(\close)\, : \, 1<r_- \r\}.
\end{align*}
For $r \in E_1$ the anisotropic $r$-Laplace differential operator is defined by
\begin{align*}
        \Delta_{r(\cdot)} u = \divergenz \l(|\nabla u|^{r(x)-2} \nabla u \r) \quad\text{for all }u \in \Wpzero{r(\cdot)}.
\end{align*}

This operator is nonhomogeneous on account of the variable exponent $r(\cdot)$. If $r(\cdot)$ is a constant function, then we have the usual $r$-Laplace differential operator. In problem \eqref{problem} we have the sum of two such anisotropic differential operators with distinct exponents. So, even in the case of constant exponents, the differential operator in \eqref{problem} is not homogeneous. This makes the study of problem \eqref{problem} more difficult. Boundary values problems driven by a combination of differential operators of different nature, such as $(p,q)$-equations, arise in many mathematical models of physical processes. We mention the works of Benci-D'Avenia-Fortunato-Pisani \cite{1-Benci-DAvenia-Fortunato-Pisani-2000}, where $(p,2)$-equations are used as a model for elementary particles in order to produce soliton-type solutions. We also mention the works of Cherfils-Il\cprime yasov \cite{3-Cherfils-Ilyasov-2005}, where the authors study the steady state solutions of reaction-diffusion systems and of Zhikov \cite{23-Zhikov-1986,24-Zhikov-2011} who studied problems related to nonlinear elasticity theory.

In the reaction of \eqref{problem} we have the competing effects of a singular term $s \to s^{-\eta(x)}$ and of a Carath\'eodory function $f\colon \Omega\times\R\to\R$, that is, $x\to f(x,s)$ is measurable for all $s\in\R$ and $s\to f(x,s)$ is continuous for a.\,a.\,$x\in\Omega$. We assume that $f(x,\cdot)$ exhibits $(p_+-1)$-superlinear growth uniformly for a.\,a.\,$x\in\Omega$ as $s\to +\infty$ but need not satisfy the Ambrosetti-Rabinowitz condition (the AR-condition for short) which is common in the literature when dealing with superlinear problems. The sum of the two terms is multiplied with a parameter $\lambda>0$.

Applying a combination of variational tools from the critical point theory along with truncation and comparison techniques, we prove a bifurcation-type theorem describing the changes in the set of positive solutions as the parameter $\lambda$ moves on the open positive semiaxis $\overset{\circ}{\R}_+=(0,+\infty)$. We also show that for every admissible parameter $\lambda>0$, problem \eqref{problem} has a smallest positive solution $\tilde{u}_\lambda$ and we determine the monotonicity and continuity properties of the minimal solution map $\lambda \to \tilde{u}_\lambda$.

Boundary values problems driven by the anisotropic $p$-Laplacian have been studied extensively in the last decade. We refer to the books of Diening-Harjulehto-H\"{a}st\"{o}-R$\mathring{\text{u}}$\v{z}i\v{c}ka \cite{4-Diening-Harjulehto-Hasto-Ruzicka-2011} and R\u{a}dulescu-Repov\v{s} \cite{19-Radulescu-Repovs-2015} and the references therein. In contrast, the study of singular anisotropic equations is lagging behind. There are very few works on this subject. We mention two such papers which are close to our problem \eqref{problem}. These are the works of Byun-Ko \cite{2-Byun-Ko-2017} and Saoudi-Ghanmi \cite{20-Saoudi-Ghanmi-2017} who examine equations driven by the anisotropic $p$-Laplacian and the parameter multiplies only the singular term. Moreover, the overall conditions on the data of the problem are more restrictive, see hypotheses ($p_M$) in \cite{2-Byun-Ko-2017} and hypotheses (H1)--(H4) in \cite{20-Saoudi-Ghanmi-2017}. Finally, we mention the isotropic works of the authors \cite{Papageorgiou-Winkert-2020,18-Papageorgiou-Winkert-2019} on singular equations driven by the $(p,q)$-Laplacian and the $p$-Laplacian, respectively.

\section{Preliminaries and Hypotheses}

In this section we recall some basic facts about Lebesgue and Sobolev spaces with variable exponents. We refer to the book of Diening-Harjulehto-H\"{a}st\"{o}-R$\mathring{\text{u}}$\v{z}i\v{c}ka \cite{4-Diening-Harjulehto-Hasto-Ruzicka-2011} for details.

Let $M(\Omega)$ be the space of all measurable functions $u\colon \Omega\to\R$. We identify two such functions when they differ only on a Lebesgue-null set. Given $r \in E_1$, the anisotropic Lebesgue space $\Lp{r(\cdot)}$ is defined by
\begin{align*}
	\Lp{r(\cdot)}=\l\{u \in M(\Omega)\,:\, \into |u|^{r(x)} \diff x<\infty \r\}.
\end{align*}
This space is equipped with the Luxemburg norm defined by
\begin{align*}
	\|u\|_{r(\cdot)} =\inf \l \{\mu>0 \, : \, \into \l(\frac{|u|}{\mu}\r)^{r(x)}\diff x \leq 1 \r\}.
\end{align*}
The modular function related to these spaces is defined by
\begin{align*}
	\varrho_{r(\cdot)}(u) =\into |u|^{r(x)}\diff x \quad\text{for all } u\in\Lp{r(\cdot)}.
\end{align*}
It is clear that $\|\cdot\|_{r(\cdot)}$ is the Minkowski functional of the set $\{u \in \Lp{r(\cdot)}\,:\,\varrho_{r(\cdot)}(u)\leq 1 \}$. The following proposition states the relation between $\|\cdot\|_{r(\cdot)}$ and the modular $\varrho_{r(\cdot)}\colon\Lp{r(\cdot)}\to\R$.

\begin{proposition}\label{proposition_1}
	Let $r\in E_1$, let $u \in\Lp{r(\cdot)}$ and let $\{u_n\}_{n\in\N}\subseteq \Lp{r(\cdot)}$. The following assertions hold:
	\begin{enumerate}
		\item[(a)]
		$\|u\|_{r(\cdot)}=\mu \quad\Longleftrightarrow\quad \varrho_{r(\cdot)}\l(\frac{u}{\mu}\r)=1$;
		\item[(b)]
		$\|u\|_{r(\cdot)}<1$ (resp. $=1$, $>1$) $\quad\Longleftrightarrow\quad \varrho_{r(\cdot)}(u)<1$ (resp. $=1$, $>1$);
		\item[(c)]
		$\|u\|_{r(\cdot)}\leq 1$ $\quad\Longrightarrow\quad$ $\|u\|_{r(\cdot)}^{r_+} \leq \varrho_{r(\cdot)}(u) \leq \|u\|_{r(\cdot)}^{r_-}$;
		\item[(d)]
		$\|u\|_{r(\cdot)}\geq 1$ $\quad\Longrightarrow\quad$ $\|u\|_{r(\cdot)}^{r_-} \leq \varrho_{r(\cdot)}(u) \leq \|u\|_{r(\cdot)}^{r_+}$;
		\item[(e)]
		$\|u_n\|_{r(\cdot)} \to 0$ (resp. $\to \infty$)  $\quad\Longleftrightarrow\quad\varrho_{r(\cdot)}(u_n)\to 0$ (resp. $\to \infty$);
		\item[(f)]
		$\|u_n-u\|_{r(\cdot)}\to 0 \quad\Longleftrightarrow\quad \varrho_{r(\cdot)}(u_n-u)\to 0$.
	\end{enumerate}
\end{proposition}

We know that $(\Lp{r(\cdot)},\|\cdot\|_{r(\cdot)})$ is a separable and reflexive Banach space. Further we denote by $r'(x)=\frac{r(x)}{r(x)-1}$ the conjugate variable exponent to $r\in E_1$, that is,
\begin{align*}
	\frac{1}{r(x)}+\frac{1}{r'(x)}=1 \quad\text{for all }x\in\close.
\end{align*}
It is clear that $r'\in E_1$. We know that $\Lp{r(\cdot)}^*=\Lp{r'(\cdot)}$ and the following version of H\"older's inequality holds
\begin{align*}
	\into |uv| \diff x \leq \l[\frac{1}{r_-}+\frac{1}{r'_-}\r] \|u\|_{r(\cdot)}\|v\|_{r'(\cdot)}
\end{align*}
for all $u\in \Lp{r(\cdot)}$ and for all $v \in \Lp{r'(\cdot)}$.

Moreover, if $r_1, r_2\in E_1$ and $r_1(x) \leq r_2(x)$ for all $x\in \close$, then we have the continuous embedding
\begin{align*}
	\Lp{r_2(\cdot)} \hookrightarrow \Lp{r_1(\cdot)}.
\end{align*}

The corresponding variable exponent Sobolev spaces can be defined in a natural way using the variable exponent Lebesgue spaces. So, given $r \in E_1$, we define 
\begin{align*}
	\Wp{r(\cdot)}=\l\{ u \in \Lp{r(\cdot)} \,:\, |\nabla u| \in \Lp{r(\cdot)}\r\}
\end{align*}
with $\nabla u$ being the gradient of $u\colon\Omega\to\R$. This space is equipped with the norm
\begin{align*}
	\|u\|_{1,r(\cdot)}=\|u\|_{r(\cdot)}+\|\nabla u\|_{r(\cdot)} \quad\text{for all } u \in \Wp{r(\cdot)}
\end{align*}
with $\|\nabla u\|_{r(\cdot)}=\|\,|\nabla u|\,\|_{r(\cdot)}$.

Let $r \in E_1$ be Lipschitz continuous, that is, $r_1 \in E_1 \cap C^{0,1}(\close)$. We define 
\begin{align*}
	\Wpzero{r(\cdot)}= \overline{C^\infty_c(\Omega)}^{\|\cdot\|_{1,r(\cdot)}}.
\end{align*}
The spaces $\Wp{r(\cdot)}$ and $\Wpzero{r(\cdot)}$ are both separable and reflexive Banach spaces. On the space $\Wpzero{r(\cdot)}$ we have the Poincar\'e inequality, namely there exists $\hat{c}>0$ such that 
\begin{align*}
	\|u\|_{r(\cdot)} \leq \hat{c} \|\nabla u\|_{r(\cdot)} \quad\text{for all } u \in \Wpzero{r(\cdot)}.
\end{align*}

Let $r \in E_1\cap C^{0,1}(\close)$ and set
\begin{align*}
	r^*(x)=
	\begin{cases}
		\frac{Nr(x)}{N-r(x)} & \text{if }r(x)<N,\\
		+\infty & \text{if } N \leq r(x),
	\end{cases} \quad\text{for all }x\in\close,
\end{align*}
which is the critical variable Sobolev exponent corresponding to $r$. Let $q\in C(\close)$ be such that $1 \leq q_- \leq q(x)\leq r^*(x)$ (resp. $1\leq q_-\leq q(x)<r^*(x)$) for all $x \in \close$. If $X=\Wp{r(\cdot)}$ or $X=\Wpzero{r(\cdot)}$, then we have 
\begin{align*}
	X \hookrightarrow \Lp{q(\cdot)} \quad \text{continuously} \quad (\text{resp.\,compactly}).
\end{align*}
This is the anisotropic Sobolev embedding theorem.

For $r\in E_1\cap C^{0,1}(\close)$, we have
\begin{align*}
	\Wpzero{r(\cdot)}^*=W^{-1,r'(\cdot)}(\Omega).
\end{align*}

Let $A_{r(\cdot)}\colon \Wpzero{r(\cdot)}\to W^{-1,r'(\cdot)}(\Omega)$ be the nonlinear operator defined by
\begin{align*}
	\l\lan A_{r(\cdot)}(u),h\r\ran= \into |\nabla u|^{r(x)-2} \nabla u \cdot \nabla h\diff x\quad\text{for all }u,h\in\Wpzero{r(\cdot)}.
\end{align*} 
This map has the following properties, see, for example Gasi\'nski-Papageorgiou \cite[Proposition 2.5]{7-Papageorgiou-Gasinski-2011} and R\u{a}dulescu-Repov\v{s} \cite[p.\,40]{19-Radulescu-Repovs-2015}.

\begin{proposition}\label{proposition_2}
	The operator $A_{r(\cdot)}\colon \Wpzero{r(\cdot)}\to W^{-1,r'(\cdot)}(\Omega)$ is bounded (so it maps bounded sets to bounded sets), continuous, strictly monotone (which implies it is also maximal monotone) and of type $(\Ss)_+$, that is,
	\begin{align*}
		u_n\weak u \text{ in }\Wpzero{r(\cdot)} \text{ and } \limsup_{n\to\infty} \l\lan A_{r(\cdot)}(u_n),u_n-u\r\ran \leq 0
	\end{align*}
	imply $u_n\to u$ in $\Wpzero{r(\cdot)}$.
\end{proposition}

The anisotropic singular regularity theory, see Saoudi-Ghanmi \cite[Appendix 2]{20-Saoudi-Ghanmi-2017}, leads to another Banach space, namely the space
\begin{align*}
	C^1_0(\close)=\l\{u\in C^1(\close)\, :\, u\big|_{\partial\Omega}=0  \r\}.
\end{align*}
This is an ordered Banach space with positive (order) cone
\begin{align*}
	C^1_0(\overline{\Omega})_+=\left\{u \in C^1_0(\overline{\Omega}): u(x) \geq 0 \text{ for all } x \in \overline{\Omega}\right\}.
\end{align*}
This cone has a nonempty interior given by
\begin{align*}
\ints \left(C^1_0(\overline{\Omega})_+\right)=\left\{u \in C^1_0(\overline{\Omega})_+: u(x)>0 \text{ for all } x \in \Omega \text{, } \frac{\partial u}{\partial n}\bigg|_{\partial\Omega}<0  \right\},
\end{align*}
where $\frac{\partial u}{\partial n}=\nabla u \cdot n$ with $n$ being the outward unit normal on $\partial \Omega$.

Our hypotheses on the exponents $p(\cdot)$, $q(\cdot)$ and $\eta(\cdot)$ are the following ones:
\begin{enumerate}
	\item[H$_0$:]
	$p,q \in E_1\cap C^{0,1}(\close)$, $\eta\in C(\close)$, $1<q_-\leq q_+<p_-\leq p_+$ and $0<\eta(x)<1$ for all $x\in \close$.
\end{enumerate}

Using these conditions on the exponents and following the arguments in the papers of Papageorgiou-R\u{a}dulescu-Repov\v{s} \cite[Proposition 2.4]{14-Papageorgiou-Radulescu-Repovs-2020}, \cite[Proposition 6]{15-Papageorgiou-Radulescu-Repovs-2020} we can have two strong comparison principles.

For the first, we will need the following ordering notion on $M(\Omega)$.

So, given $y_1,y_2\colon \Omega \to\R$ two measurable functions, we write $y_1 \preceq y_2$ if for every compact set $K\subseteq \Omega$, we have $0<c_K\leq y_2(x)- y_1(x)$ for a.\,a.\,$x\in K$. Note that if $y_1,y_2\in C(\Omega)$ and $y_1(x)< y_2(x)$ for all $x\in \Omega$, then $y_1 \preceq y_2$. The first strong comparison principle is the following one.
\begin{proposition}\label{proposition_3}
	If hypotheses H$_0$ hold, $\hat{\xi} \in \Linf$, $\hat{\xi}(x) \geq 0$ for a.\,a.\,$x\in \Omega$, $y_1, y_2\in \Linf$, $y_1\preceq y_2$, $u\in \Wp{p(\cdot)}$, $u\geq 0$ for a.\,a.\,$x\in \Omega$, $u\neq 0$, $v\in\interior$ and
	\begin{align*}
		-\Delta_{p(\cdot)}u-\Delta_{q(\cdot)}u+\hat{\xi} (x) u^{p(x)-1}-u^{-\eta(x)}=y_1(x) & \text{ in }\Omega,\\
		-\Delta_{p(\cdot)}v-\Delta_{q(\cdot)}v+\hat{\xi} (x) v^{p(x)-1}-v^{-\eta(x)}=y_2(x) & \text{ in }\Omega,
	\end{align*}
	then $v-u \in\interior$.
\end{proposition}

In the second strong comparison principle, we strengthen the order condition on $y_1$ and $y_2$ but drop the boundary requirements on $u$ and $v$.
\begin{proposition}\label{proposition_4}
	If hypotheses H$_0$ hold, $\hat{\xi} \in \Linf$, $\hat{\xi} \geq 0$ for a.\,a.\,$x\in \Omega$, $y_1, y_2\in \Linf$, $0<c_0\leq y_2(x)-y_1(x)$ for a.\,a.\,$x\in\Omega$, $u,v\in C^{1,\alpha}(\close)$, $0<u(x) \leq v(x)$ for all $x\in \Omega$ and
	\begin{align*}
		-\Delta_{p(\cdot)}u-\Delta_{q(\cdot)}u+\hat{\xi} (x) 	u^{p(x)-1}-u^{-\eta(x)}=y_1(x) & \text{ in }\Omega,\\
		-\Delta_{p(\cdot)}v-\Delta_{q(\cdot)}v+\hat{\xi} (x) v^{p(x)-1}-v^{-\eta(x)}=y_2(x) & \text{ in }\Omega,
	\end{align*}
	then $u(x)<v(x)$ for all $x\in\Omega$.
\end{proposition}

Given $u\in M(\Omega)$, we define $u^{\pm}=\max \{\pm u,0\}$ being the positive and negative part of $u$, respectively. We know that $u=u^+-u^-$, $|u|=u^++u^-$ and if $u\in \Wpzero{p(\cdot)}$, then $u^{\pm}\in\W$.

If $u,v\in M(\Omega)$ and $u(x)\leq v(x)$ for a.\,a.\,$x\in\Omega$, then we define \begin{align*}
	[u,v]&=\left\{h\in\Wpzero{p(\cdot)}: u(x) \leq h(x) \leq v(x) \text{ for a.\,a.\,}x\in\Omega\right\},\\[1ex]
	[u) & = \left\{h\in \Wpzero{p(\cdot)}: u(x) \leq h(x) \text{ for a.\,a.\,}x\in\Omega\right\}.
\end{align*}
Moreover, we denote by $\sideset{}{_{C^1_0(\close)}} \ints [u,v]$  the interior of $[u,v]\cap C^1_0(\close)$ in $C^1_0(\close)$.

In what follows, for notational simplicity, we denote by $\|\cdot\|$ the norm of the anisotropic Sobolev space $\W$. On account of Poincar\'e's inequality we have
\begin{align*}
	\|u\|=\|\nabla u\|_{r(\cdot)} \quad\text{for all } u \in \Wpzero{r(\cdot)}.
\end{align*}

Given a Banach space $X$ and a functional $\ph \in C^1(X)$, we define
\begin{align*}
	K_\ph=\left\{ u\in X: \ph'(u)=0\right\}
\end{align*}
being the critical set of $\ph$. We say that $\ph$ satisfies the ``Cerami condition'', C-condition for short, if every sequence $\{u_n\}_{n\in\N}\subseteq X$ such that $\{\ph(u_n)\}_{n\in\N}\subseteq \R$ is bounded and 
\begin{align*}
	\left(1+\|u_n\|_X\right)\ph'(u_n) \to 0\quad\text{in }X^* \text{ as }n\to \infty,
\end{align*}    
admits a strongly convergent subsequence. This is a compactness-type condition on the functional $\ph$ which compensates for the fact that the ambient space $X$ is not locally compact in general, since it could be infinite dimensional. Using this condition, we can prove a deformation theorem which leads to the minimax theorems of the critical point theory, see, for example, Papageorgiou-R\u{a}dulescu-Repov\v{s} \cite[Section 5.4]{13-Papageorgiou-Radulescu-Repovs-2019}.

Now we are ready to state our hypotheses on the nonlinearity $f\colon\Omega\times\R\to\R$.
\begin{enumerate}
	\item[H$_1$:]
		$f\colon\Omega\times\R\to\R$ is a Carath\'eodory function such that $f(x,0)=0$ for a.\,a.\,$x\in\Omega$ and
		\begin{enumerate}
			\item[(i)]
				there exists $a \in \Linf$ such that
				\begin{align*}
					0\leq f(x,s) \leq a(x) \l [ 1+s^{r(x)-1}\r]
				\end{align*}
				for a.\,a.\,$x\in \Omega$, for all $s\geq 0$ with $r \in C(\close)$ such that $p_+<r_-\leq r(x)<p^*(x)$ for all $x \in \close$;
			\item[(ii)]
				if $F(x,s)=\displaystyle\int_0^s f(x,t)\diff t$, then
				\begin{align*}
					\lim_{s\to +\infty} \frac{F(x,s)}{s^{p_+}}=+\infty \quad\text{uniformly for a.\,a.\,}x\in\Omega;
				\end{align*}
			\item[(iii)]
				there exists a function $\mu\in C(\close)$ such that
				\begin{align*}
					\mu(x) \in \l (\l(r_+-p_-\r)\max\l\{\frac{N}{p_-},1\r\}, p^*(x) \r)\quad\text{ for all }x\in\close
				\end{align*}
				and
				\begin{align*}
					0<\hat{\eta}_0 \leq \liminf_{s\to +\infty} \frac{f(x,s)s-p_+F(x,s)}{s^{\mu(x)}}
				\end{align*}
				uniformly for a.\,a.\,$x\in\close$;
			\item[(iv)]
				\begin{align*}
					0 < \hat{\eta}_1 \leq \liminf_{s\to 0^+}\frac{f(x,s)}{s^{q_+-1}}\quad\text{uniformly for a.\,a.\,}x\in\Omega
				\end{align*}
				and for every $\ell>0$ there exists $m_\ell>0$ such that $m_\ell \leq f(x,s)$ for a.\,a.\,$x\in\Omega$ and for all $s \geq \ell$;
			\item[(v)]
				for every $\rho>0$ there exists $\hat{\xi}_\rho>0$ such that the function
				\begin{align*}
					s\to f(x,s)+\hat{\xi}_\rho s^{p(x)-1}
				\end{align*}
				is nondecreasing on $[0,\rho]$ for a.\,a.\,$x\in \Omega$.
		\end{enumerate}
\end{enumerate}

\begin{remark}
	Without any loss of generality we can assume that $f(x,s)= 0$ for a.\,a.\,$x\in\Omega$ and for all $s \leq 0$ since we are interested in positive solutions of \eqref{problem}. Hypotheses H$_1$(ii), (iii) imply that $f(x,\cdot)$ is $(p_+-1)$-superlinear for a.\,a.\,$x\in\Omega$. In most papers in the literature, superlinear problems are treated by using the AR-condition which in the present context has the following form:
	\begin{enumerate}
		\item[(AR)$_+$:]
			There exist $\theta>p_+$ and $M>0$ such that
			\begin{align}
				0&<\theta F(x,s) \leq f(x,s)s\quad\text{for a.\,a.\,}x\in\Omega \text{ and for all }s\geq M,\label{1a}\\
				0&<\essinf_{x\in\Omega}F(x,M)\label{1b}.
			\end{align}
	\end{enumerate}
	This is a unilateral version of the AR-condition since we assumed that $f(x,s)=0$ for a.\,a.\,$x\in\Omega$ and for all $s\leq 0$. Integrating \eqref{1a} and using \eqref{1b} gives
	\begin{align*}
		c_1 s^\theta \leq F(x,s)
	\end{align*}
	for a.\,a.\,$x\in\Omega$, for all $s\geq M$ and for some $c_1>0$. Hence,
	\begin{align*}
		c_1 s^{\theta-1} \leq f(x,s)
	\end{align*}
	for a.\,a.\,$x\in\Omega$ and for all $s\geq M$, see \eqref{1a}. Therefore, the (AR)$_+$-condition dictates that $f(x,\cdot)$ has at least $(\theta-1)$-polynomial growth as $s\to +\infty$. But this way we exclude superlinear nonlinearities with ``slower'' growth near $+\infty$ from our considerations. The following example fulfills H$_1$, but fails to satisfy the (AR)$_+$-condition:
	\begin{align*}
		f(x,s)=
		\begin{cases}
			\l(s^+\r)^{\tau(x)-1} &\text{if } s\leq 1,\\
			s^{p_+-1}\ln(x)+s^{\theta(x)-1}&\text{if } 1<s
		\end{cases}
	\end{align*}
	with $\tau\in C(\close)$, $\tau_+\leq q_+$ and $\theta\in C(\close)$ such that $\theta_+\leq p_+$.
	
	Hypotheses H$_1$(iv) implies that $f(x,\cdot)$ is strictly $(q_+-1)$-sublinear.
\end{remark}

When studying singular problems of isotropic and anisotropic type, the presence of the singular term leads to an energy function which is not $C^1$ and so we cannot apply directly the minimax theorems of the critical point theory on it. We need to find a way to bypass the singularity and deal with $C^1$-functionals. To this end, we examine a purely singular problem in the next section. The unique solution of this problem will be helpful in our effort to bypass the singularity of our original problem \eqref{problem}.

\section{An auxiliary purely singular problem}

In this section we study the following purely singular anisotropic Dirichlet problem
\begin{align}\tag{Au$_\lambda$}\label{problem_aux}
	-\Delta_{p(\cdot)} u - \Delta_{q(\cdot)} u = \lambda u^{-\eta(x)} \quad \text{in }\Omega, \quad u\big|_{\partial\Omega}=0, \quad u>0, \quad \lambda>0.
\end{align}

We have the main result in this section.

\begin{proposition}\label{proposition_5}
	If hypotheses H$_0$ hold, then problem \eqref{problem_aux} has a unique positive solution $\overline{u}_\lambda\in \interior$. Moreover, the mapping $\lambda \to \overline{u}_\lambda$ is nondecreasing, that is, if $0<\lambda'<\lambda$, then we have $\overline{u}_{\lambda'}\leq \overline{u}_\lambda$.
\end{proposition}

\begin{proof}
	For the existence and uniqueness part of the proof we assume for simplicity that $\lambda=1$. 
	
	To this end, let $g \in \Lp{p(\cdot)}$ and let $\eps \in (0,1]$. We consider the following Dirichlet problem
	\begin{align}\label{2}
	-\Delta_{p(\cdot)} u - \Delta_{q(\cdot)} u = \big[|g(x)|+\eps\big]^{-\eta(x)} \quad \text{in }\Omega, \quad u\big|_{\partial\Omega}=0, \quad u>0.
	\end{align}

	Let $V\colon \Wpzero{p(\cdot)}\to \Wpzero{p(\cdot)}^*=W^{-1,p'(\cdot)}(\Omega)$ be the nonlinear operator defined by
	\begin{align*}
		V(u)=A_{p(\cdot)}(u)+A_{q(\cdot)}(u) \quad\text{ for all } u\in\Wpzero{p(\cdot)}.
	\end{align*}
	This operator is bounded, continuous, strictly monotone and so maximal monotone, see Proposition \ref{proposition_2}. It is clear that it is also coercive, see Proposition \ref{proposition_1}. From Corollary 2.8.7 of Papageorgiou-R\u{a}dulescu-Repov\v{s} \cite[p.\,135]{13-Papageorgiou-Radulescu-Repovs-2019} we know that $V\colon \Wpzero{p(\cdot)}\to W^{-1,p'(\cdot)}(\Omega)$ is surjective. Since $[|g(\cdot)|+\eps]^{-\eta(\cdot)}\in \Linf$ we can find  $\hat{v}_\eps \in \Wpzero{p(\cdot)}$ such that
	\begin{align*}
		V(\hat{v}_\eps)=\big[|g|+\eps\big]^{-\eta(\cdot)}\quad \text{in }W^{-1,p'(\cdot)}(\Omega).
	\end{align*}
	From the strict monotonicity of $V$ we know that $\hat{v}_\eps$ is the unique solution of \eqref{2}. Moreover, by acting with $-\hat{v}_\eps^-\in \W$ we obtain $\hat{v}_\eps\geq 0$ and $\hat{v}_\eps\not\equiv 0$. Thus, we have
	\begin{align*}
		-\Delta_{p(\cdot)} \hat{v}_\eps - \Delta_{q(\cdot)} \hat{v}_\eps = \big[|g|+\eps\big]^{-\eta(x)} \quad \text{in }\Omega, \quad \hat{v}_\eps\big|_{\partial\Omega}=0.
	\end{align*}
	
	Theorem 4.1 of Fan-Zhao \cite{5-Fan-Zhao-1999} implies that $\hat{v}_\eps\in\Linf$. Invoking Corollary 1.1 of Tan-Fang \cite{21-Tan-Fang-2013} (see also Lemma 3.3 of Fukagai-Narukawa \cite{6-Fukagai-Narukawa-2007}) we have that $\hat{v}_\eps\in C^1_0(\Omega)_+\setminus \{0\}$. Finally, the anisotropic maximum principle of Zhang \cite[Theorem 1.2]{22-Zhang-2005} says that $\hat{v}_\eps \in \interior$.
	
	Now we can define the solution map $K_\eps \colon \Lp{p(\cdot)}\to \Lp{p(\cdot)}$ given by
	\begin{align*}
		K_\eps(g)=\hat{v}_\eps.
	\end{align*}
	We have
	\begin{align}\label{3}
		\l\lan A_{p(\cdot)}\l(\hat{v}_\eps\r),h\r\ran + \l\lan A_{q(\cdot)}\l(\hat{v}_\eps\r),h\r\ran =\into \frac{h}{\l[|g|+\eps\r]^{\eta(x)}}\diff x
	\end{align}
	for all $h \in \W$. Choosing $h=\hat{v}_\eps=K_\eps(g)\in\W$ in \eqref{3} gives
	\begin{align*}
		\varrho_{p(\cdot)}\l(\nabla \hat{v}_\eps\r)+\varrho_{q(\cdot)}\l(\nabla \hat{v}_\eps\r) \leq c_\eps \l\|\hat{v}_\eps \r\|\quad\text{for some }c_\eps>0.
	\end{align*}
	Assume that  $\l\|\hat{v}_\eps \r\|\geq 1$, then, by Proposition \ref{proposition_1}, one gets
	\begin{align*}
		\l\|\hat{v}_\eps \r\|^{p_-}\leq c_\eps \l\|\hat{v}_\eps \r\|
	\end{align*}
	and so,
	\begin{align}\label{4}
		\l\|\hat{v}_\eps \r\|^{p_--1}=\l\| K_\eps\l(g\r)\r\|^{p_--1}\leq c_\eps \quad\text{for all }g\in \Lp{p(\cdot)}.
	\end{align}
	It follows that $K_\eps \colon \Lp{p(\cdot)}\to \Lp{p(\cdot)}$ maps $\LP$ onto a bounded subset of $\W$.
	
	{\bf Claim 1:} $K_\eps \colon \Lp{p(\cdot)}\to \Lp{p(\cdot)}$ is continuous.

	Let $g_n\to g$ in $\Lp{p(\cdot)}$ and let $\hat{v}_n=K_\eps(g_n)$ with $n\in\N$. From \eqref{4} we know that
	\begin{align*}
		\l\{\hat{v}_n\r\}_{n\in\N}=\l\{K_\eps(g_n)\r\}_{n\in\N}\subseteq \W \quad \text{is bounded}.
	\end{align*}
	
	We may assume that
	\begin{align}\label{5}
		\hat{v}_n \weak \hat{v} \quad\text{in }\Wpzero{p(\cdot)} 
		\quad\text{and}\quad
		\hat{v}_n \to \hat{v} \quad\text{in }\Lp{p(\cdot)}.
	\end{align}
	We have
	\begin{align}\label{6}
		\l\lan A_{p(\cdot)}\l(\hat{v}_n\r),h\r\ran + \l\lan A_{q(\cdot)}\l(\hat{v}_n\r),h\r\ran =\into \frac{h}{\l[|g_n|+\eps\r]^{\eta(x)}}\diff x
	\end{align}
	for all $h\in \Wpzero{p(\cdot)}$ and for all $n\in\N$. We choose $h= \hat{v}_n-\hat{v} \in \Wpzero{p(\cdot)}$ in \eqref{6}, pass to the limit as $n\to \infty$ and use \eqref{5} and the fact that 
	\begin{align*}
		\l\{\frac{1}{\l(|g_n|+\eps\r)^{\eta(\cdot)}}\r\}_{n\in\N}\subseteq \Linf \quad \text{is bounded}.
	\end{align*}
	This yields
	\begin{align*}
		\lim_{n\to\infty} \l [\l\lan A_{p(\cdot)}\l(\hat{v}_n\r),\hat{v}_n-\hat{v} \r\ran+\l\lan A_{q(\cdot)}\l(\hat{v}_n\r),\hat{v}_n-\hat{v} \r\ran\r]=0.
	\end{align*}
	Due to the monotonicity of $A_{q(\cdot)}(\cdot)$ we obtain
	\begin{align*}
		\limsup_{n\to\infty} \l [\l\lan A_{p(\cdot)}\l(\hat{v}_n\r),\hat{v}_n-\hat{v} \r\ran+\l\lan A_{q(\cdot)}\l(\hat{v}\r),\hat{v}_n-\hat{v}\r\ran\r]\leq 0.
	\end{align*}
	From this and \eqref{5} we then conclude that
	\begin{align*}
		\limsup_{n\to\infty} \l\lan A_{p(\cdot)}\l(\hat{v}_n\r),\hat{v}_n-\hat{v} \r\ran \leq 0.
	\end{align*}
	By the $(\Ss)_+$-property of $A_{p(\cdot)}$, see Proposition \ref{proposition_2}, we have that
	\begin{align}\label{7}
		\hat{v}_n \to \hat{v} \quad \text{in }\Wpzero{p(\cdot)}.
	\end{align}
	Passing  to the limit in \eqref{6} as $n\to \infty$ and using \eqref{7} gives
	\begin{align*}
		\l\lan A_{p(\cdot)}\l(\hat{v}\r),h\r\ran + \l\lan A_{q(\cdot)}\l(\hat{v}\r),h\r\ran =\into \frac{h}{\l[|g|+\eps\r]^{\eta(x)}}\diff x
	\end{align*}
	for all $h\in \Wpzero{p(\cdot)}$. Thus, $\hat{v} =K_\eps(g)$.  Therefore, by the Urysohn criterion for the convergence of sequences, we conclude that for the original sequence we have
	\begin{align*}
		\hat{v}_n=K_\eps(g_n) \to K_\eps (g) =\hat{v}.
	\end{align*}
	Hence, $K_\eps$ is continuous and this proves Claim 1.
	
	Recall that $K_\eps(\LP)\subseteq \W$ is bounded, see \eqref{4}. On the other hand, we have the compact embedding $\W \hookrightarrow \LP$. This implies
	\begin{align}\label{8}
		\overline{K_\eps(\LP)}^{\|\cdot\|_{p(\cdot)}} \subseteq \LP \quad\text{is compact}.
	\end{align}
	
	Claim 1 and \eqref{8} permit the use of the Schauder-Tychonov fixed point theorem, see Papageorgiou-R\u{a}dulescu-Repov\v{s} \cite[Theorem 4.3.21]{13-Papageorgiou-Radulescu-Repovs-2019}. So, we can find $\overline{u}_\eps\in\W$ such that
	\begin{align*}
		K_\eps\l(\overline{u}_\eps\r)=\overline{u}_\eps\subseteq \interior.
	\end{align*}
	Hence
	\begin{align}\label{9}
		-\Delta_{p(\cdot)}\overline{u}_\eps-\Delta_{q(\cdot)}\overline{u}_\eps=\big[\overline{u}_\eps+\eps \big]^{-\eta(x)}\quad\text{in }\Omega, \quad \overline{u}_\eps\big|_{\partial\Omega}=0.
	\end{align}

	In fact, this solution is unique. Indeed, suppose that $\overline{y}_\eps\in\interior$ is another positive solution of \eqref{9}. Then we have that
	\begin{align*}
		0&\leq \l\lan A_{p(\cdot)}\l(\overline{u}_\eps\r)-A_{p(\cdot)}\l(\overline{y}_\eps\r),\l(\overline{u}_\eps-\overline{y}_\eps\r)^+\r\ran + \l\lan A_{q(\cdot)}\l(\overline{u}_\eps\r)-A_{q(\cdot)}\l(\overline{y}_\eps\r),\l(\overline{u}_\eps-\overline{y}_\eps\r)^+\r\ran\\
		& =\into \l[ \frac{1}{\l[\overline{u}_\eps+\eps\r]^{\eta(x)}}-\frac{1}{\l[\overline{y}_\eps+\eps\r]^{\eta(x)}}\r] \l(\overline{u}_\eps-\overline{y}_\eps\r)^+\diff x \leq 0.
	\end{align*}
	We obtain $\overline{u}_\eps\leq \overline{y}_\eps$.
	
	Interchanging the roles of $\overline{u}_\eps$ and $\overline{y}_\eps$ in the argument above also gives $\overline{y}_\eps\leq \overline{v}_\eps$. Hence, $\overline{u}_\eps=\overline{y}_\eps$. This proves the uniqueness of the solution $\overline{u}_\eps\in\interior$ of problem \eqref{9}.

	{\bf Claim 2:} If $0<\eps'\leq \eps$, then $\overline{u}_\eps \leq \overline{u}_{\eps'}$.

	First note that $\overline{u}_\eps, \overline{u}_{\eps'}\in\interior$. Since $\eps'\leq \eps$ we have
	\begin{align}\label{11}
		-\Delta_{p(\cdot)} \overline{u}_{\eps'} -\Delta_{q(\cdot)} \overline{u}_{\eps'}=
		\l[\overline{u}_{\eps'}+\eps'\r]^{-\eta(x)}\geq \l[\overline{u}_{\eps'}+\eps\r]^{-\eta(x)}\quad\text{in }\Omega.
	\end{align}

	Next we introduce the Carath\'eodory function $l_\eps\colon \Omega\times\R\to\R$ defined by
	\begin{align}\label{12}
		l_\eps(x,s)=
		\begin{cases}
			\l[s^++\eps\r]^{-\eta(x)} &\text{if } s \leq \overline{u}_{\eps'}(x),\\
			\l[\overline{u}_{\eps'}(x)+\eps\r]^{-\eta(x)} & \text{if } \overline{u}_{\eps'}(x) < s.
		\end{cases}
	\end{align}

	Let $L_\eps(x,s)=\int^s_0 l_\eps(x,t)\diff t$ and consider the $C^1$-functional $\psi_\eps\colon \Wpzero{p(\cdot)}\to \R$ defined by
	\begin{align*}
		\psi_\eps(u)=\into \frac{1}{p(x)} |\nabla u|^{p(x)}\diff x +\into \frac{1}{q(x)} |\nabla u|^{q(x)}\diff x-\into L_\eps(x,u)\diff x
	\end{align*}
	for all $u \in \Wpzero{p(\cdot)}$. From the definition of the truncation in \eqref{12} we see that
	\begin{align*}
		\psi_\eps(u) \geq \frac{1}{p_+} \l[\varrho_{p(\cdot)}(\nabla u)+\varrho_{q(\cdot)}(\nabla u) \r]-c_1
	\end{align*}
	for some $c_1>0$. Hence, $\psi_\eps\colon \Wpzero{p(\cdot)}\to \R$ is coercive. Moreover, by the anisotropic Sobolev embedding theorem we know that $\psi_\eps\colon \Wpzero{p(\cdot)}\to \R$ is sequentially weakly lower semicontinuous. Then, by the Weierstra\ss-Tonelli theorem, we can find $\tilde{u}_\eps\in\Wpzero{p(\cdot)}$ such that 
	\begin{align}\label{13}
		\psi_\eps\l(\tilde{u}_\eps\r)=\min \l[\psi_\eps(u)\,:\,u \in \Wpzero{p(\cdot)}\r].
	\end{align}
	
	Let $u\in \interior$ be fixed. Since $\overline{u}_{\eps'}\in\interior$, we can take $t \in (0,1)$ small enough such that $tu \leq\overline{u}_{\eps'}$, see also Proposition 4.1.22 of Papageorgiou-R\u{a}dulescu-Repov\v{s} \cite{13-Papageorgiou-Radulescu-Repovs-2019}. From \eqref{12} we see that have that $\psi_\eps(tu)<0$ and so $\psi_\eps\l(\tilde{u}_\eps\r)<0=\psi_\eps(0)$. Hence, $\tilde{u}_\eps\neq 0$.
	
	Taking \eqref{13} into account we have $\psi'_\eps\l(\tilde{u}_\eps\r)=0$, that is,
	\begin{align}\label{14}
		\l\lan A_{p(\cdot)}\l(\tilde{u}_\eps\r),h\r\ran+\l\lan A_{q(\cdot)}\l(\tilde{u}_\eps\r),h\r\ran=\into l_\eps \l(x,\tilde{u}_\eps\r)h\diff x
	\end{align}
	for all $h\in \Wpzero{p(\cdot)}$. First we test \eqref{14} with $h=-\l(\tilde{u}_\eps\r)^-\in\Wpzero{p(\cdot)}$ in order to get
	\begin{align*}
		\varrho_{p(\cdot)}\l(\nabla \l(\tilde{u}_\eps\r)^-\r)+\varrho_{q(\cdot)}\l(\nabla \l(\tilde{u}_\eps\r)^-\r) \leq 0.
	\end{align*}
	Proposition \ref{proposition_1} then implies that
	\begin{align*}
		\tilde{u}_\eps \geq 0 \quad\text{and}\quad \tilde{u}_\eps\neq 0.
	\end{align*}
	
	Next, we test \eqref{14} with $h= \l(\tilde{u}_\eps-\overline{u}_{\eps'}\r)^+ \in\Wpzero{p(\cdot)}$. This yields, by applying \eqref{12} and \eqref{11},
	\begin{align*}
		&\l\lan A_{p(\cdot)}\l(\tilde{u}_\eps\r),\l(\tilde{u}_\eps-\overline{u}_{\eps'}\r)^+\r\ran+\l\lan A_{q(\cdot)}\l(\tilde{u}_\eps\r),\l(\tilde{u}_\eps-\overline{u}_{\eps'}\r)^+\r\ran\\
		&=\into \frac{\l(\tilde{u}_\eps-\overline{u}_{\eps'}\r)^+}{\l[\overline{u}_{\eps'}+\eps\r]^{\eta(x)}} \diff x\\
		& \leq \l\lan A_{p(\cdot)}\l(\overline{u}_{\eps'}\r),\l(\tilde{u}_\eps-\overline{u}_{\eps'}\r)^+\r\ran+\l\lan A_{q(\cdot)}\l(\overline{u}_{\eps'}\r),\l(\tilde{u}_\eps-\overline{u}_{\eps'}\r)^+\r\ran.
	\end{align*}
	This implies $\tilde{u}_\eps \leq \overline{u}_{\eps'}$ and so it holds
	\begin{align}\label{15}
		\tilde{u}_\eps \in \l[0,\overline{u}_{\eps'}\r], \quad \tilde{u}_\eps \neq 0.
	\end{align}
	
	From \eqref{15}, \eqref{12}, \eqref{14} it follows that $\tilde{u}_\eps$ is a positive solution of problem \eqref{9}. Hence, $\tilde{u}_\eps=\overline{u}_\eps\in\interior$. Then, with view to \eqref{15}, we have 
	\begin{align*}
		\overline{u}_\eps \leq \overline{u}_{\eps'} \quad\text{for all } 0<\eps'\leq \eps.
	\end{align*}
	This proves Claim 2.
	
	Now we will let $\eps \to 0^+$ to get a solution of the purely singular problem \eqref{problem_aux}. 
	
	So, let $\eps_n\to 0^+$ and let $\overline{u}_n=\overline{u}_{\eps_n}\in\interior$ be the unique solution of problem \eqref{9} for $n \in \N$. From Claim 2 we have
	\begin{align*}
		0 \leq \overline{u}_1 \leq \overline{u}_n\quad\text{for all }n \in \N.
	\end{align*}
	It follows that
	\begin{align}\label{16}
		0 \leq \frac{1}{\l[\overline{u}_n+\eps_n\r]^{\eta(x)}} \leq \frac{1}{\overline{u}_n^{\eta(x)}} \leq \frac{1}{\overline{u}_1^{\eta(x)}}\quad\text{for all }n\in\N.
	\end{align}
	
	Since $\overline{u}_n \in \interior$ is a solution of \eqref{9}, we have
	\begin{align}\label{17}
		\l\lan A_{p(\cdot)}\l(\overline{u}_n\r),h\r\ran+\l\lan A_{q(\cdot)}\l(\overline{u}_n\r),h\r\ran
		=\into \frac{h}{\l[\overline{u}_n+\eps_n\r]^{\eta(x)}}\diff x
	\end{align}
	for all $h \in \Wpzero{p(\cdot)}$ and for all $n \in \N$. We choose $h= \overline{u}_n\in\Wpzero{p(\cdot)}$ in \eqref{17} which by using \eqref{16} gives
	\begin{align}\label{18}
		\varrho_{p(\cdot)}\l(\nabla \overline{u}_n\r)+\varrho_{q(\cdot)}\l(\nabla \overline{u}_n\r) \leq \into\frac{\overline{u}_n}{\overline{u}_1^{\eta(x)}}\diff x \quad\text{for all }n \in \N.
	\end{align}
	From Lemma 14.16 of Gilbarg-Trudinger \cite[p.\,355]{9-Gilbarg-Trudinger-1998} we know that there exists $\delta_0>0$ such that $\hat{d}(\cdot)=\hat{d}(\cdot,\partial\Omega)\in C^2(\overline{\Omega}_{\delta_0})$ with $\overline{\Omega}_{\delta_0} =\{x\in \close\,:\, \hat{d}(x)<\delta_0\}$. Hence, $\hat{d}\in C^1_0(\close))_+\setminus\{0\}\}$ and so there exists $c_2>0$ such that $c_2\hat{d} \leq \overline{u}_1$ since $\overline{u}_1 \in \interior$. Then, from \eqref{16} and \eqref{18} we obtain
	\begin{align}\label{19}
		\varrho_{p(\cdot)}\l(\nabla \overline{u}_n\r) \leq c_3\l\|\overline{u}_n\r\|
	\end{align}
	for some $c_3>0$ and for all $n\in\N$. This inequality follows from the anisotropic Hardy's inequality due to Harjulehto-H\"{a}st\"{o}-Koskenoja \cite{10-Harjulehto-Hasto-Koskenoja} and the Poincar\'e inequality. Then \eqref{19} and Proposition \ref{proposition_1} imply that $\{\overline{u}_n\}_{n\in\N}\subseteq \Wpzero{p(\cdot)}$ is bounded.
	
	From Lemma A.5 of Saoudi-Ghanmi \cite{20-Saoudi-Ghanmi-2017} it follows that $\{\overline{u}_n\}_{n\in\N}\subseteq\Linf$ is bounded and so using Lemma 3.3 of Fukagai-Narukawa \cite{6-Fukagai-Narukawa-2007}, we can find $\alpha \in (0,1)$ and $c_4>0$ such that 
	\begin{align}\label{20}
		\overline{u}_n \in C^{1,\alpha}_0(\close)=C^{1,\alpha}(\close)\cap C^1_0(\close)
		\quad\text{and}\quad
		\l\|\overline{u}_n\r\|_{C^{1,\alpha}_0(\close)} \leq c_4
	\end{align}
	for all $n \in \N$.  
	
	We know that $C^{1,\alpha}_0(\close)\hookrightarrow C^1_0(\close)$ is compactly embedded. So, from \eqref{20} and by passing to a subsequence if necessary, we may assume that
	\begin{align}\label{21}
		\overline{u}_n \to \overline{u}_\lambda\quad\text{in }C^1_0(\close).
	\end{align}
	Hence, $\overline{u}_\lambda \geq \overline{u}_1$ and so $\overline{u}_\lambda \in \interior$.
	
	From the anisotropic Hardy's inequality, see Harjulehto-H\"{a}st\"{o}-Koskenoja \cite{10-Harjulehto-Hasto-Koskenoja}, we know that 
	\begin{align*}
		\frac{|h|}{\overline{u}_1^{\eta(\cdot)}}\in \Lp{1} \quad\text{for all }h\in\W.
	\end{align*}
	From \eqref{16} we then see that
	\begin{align*}
		\l\{\frac{h}{\l[\overline{u}_n+\eps_n\r]^{\eta(\cdot)}}\r\}_{n\in\N}\subseteq \Lp{1} \text{ is uniformly integrable}
	\end{align*}
	for all $h\in\W$. Moreover, we have
	\begin{align*}
		\frac{h}{\l[\overline{u}_n+\eps_n\r]^{\eta(x)}}\longrightarrow \frac{h}{\l[\overline{u}_\lambda+\eps_n\r]^{\eta(x)}}\quad\text{for a.\,a.\,}x\in\Omega.
	\end{align*}
	So, from Vitali's theorem, see Papageorgiou-Winkert \cite[Theorem 2.3.44]{17-Papageorgiou-Winkert-2018}, we obtain
	\begin{align}\label{22}
		\into \frac{h}{\l[\overline{u}_n+\eps_n\r]^{\eta(x)}}\diff x\longrightarrow \into \frac{h}{\overline{u}_\lambda^{\eta(x)}}\diff x
	\end{align}
	for all $h \in \W$. Therefore, if we pass to the limit as $n\to\infty$ in \eqref{17} and use \eqref{21} as well as \eqref{22}, one gets
	\begin{align*}
		\l\lan A_{p(\cdot)}\l(\overline{u}_\lambda\r),h\r\ran
		+\l\lan A_{q(\cdot)}\l(\overline{u}_\lambda\r),h\r\ran
		= \into \frac{h}{\overline{u}_\lambda^{\eta(x)}}\diff x\quad\text{for all }h\in\Wpzero{p(\cdot)}.
	\end{align*}
	This shows that $\overline{u}_\lambda\in\interior$ is a positive solution of \eqref{problem_aux} for $\lambda>0$.
	
	As before, exploiting the strict monotonicity of $s\to s^{-\eta(x)}$ on $\overset{\circ}{\R}_+=(0,+\infty)$, we show that this solution $\overline{u}_\lambda\in\interior$ is unique.
	
	An argument similar to that of Claim 2 show that $0<\lambda'<\lambda$ implies $\overline{u}_{\lambda'}\leq \overline{u}_\lambda$. This finishes the proof of the proposition.
\end{proof}

\section{Positive solutions}

We introduce the following two sets
\begin{align*}
	\mathcal{L}&=\left\{\lambda>0: \text{problem \eqref{problem} has a positive solution}\right\},\\
	\mathcal{S}_\lambda&=\left\{u: u\text{ is a positive solution of problem \eqref{problem}}\right\}.
\end{align*}

First we show that the set $\mathcal{L}$ of admissible parameters is nonempty and we determine the regularity properties of the elements of $\mathcal{S}_\lambda$ for $\lambda\in\mathcal{L}$.

Let $\overline{u}_1 \in\interior$ be the unique positive solution of \eqref{problem_aux} with $\lambda=1$, see Proposition \ref{proposition_5}. From the proof of the Lemma of Lazer-McKenna \cite[p.\,274]{12-Lazer-McKenna-1991} we know that $\overline{u}_1(\cdot)^{-\eta(\cdot)}\in \Lp{1}$. We consider the following anisotropic Dirichlet problem
\begin{align}\tag*{(Au)'}\label{problem_aux2}
	-\Delta_{p(\cdot)} u - \Delta_{q(\cdot)} u = 1+\overline{u}_1^{-\eta(x)} \quad \text{in }\Omega, \quad u\big|_{\partial\Omega}=0, \quad u>0.
\end{align}

\begin{proposition}\label{proposition_6}
	If hypotheses H$_0$ hold, then problem \ref{problem_aux2} has a unique positive solution $\tilde{u}\in\interior$ such that $\overline{u}_1 \leq \tilde{u}$.
\end{proposition}

\begin{proof}
	In order to establish the existence of a positive solution, we argue as in the first part of the proof of Proposition \ref{proposition_5}. So, we consider the approximation
	\begin{align*}
		-\Delta_{p(\cdot)} u - \Delta_{q(\cdot)} u = 1+\l[\overline{u}_1+\frac{1}{n}\r]^{-\eta(x)} \quad \text{in }\Omega, \quad u\big|_{\partial\Omega}=0, \quad n\in\N.
	\end{align*}
	This problem has a unique solution $\tilde{u}_n\in\interior$. Testing the equation with $\tilde{u}_n$ we obtain
	\begin{align*}
		\varrho_{p(\cdot)}\l(\nabla \tilde{u}_n\r) \leq \into \tilde{u}_n\diff x+\into \frac{\tilde{u}_n}{\overline{u}_1^{\eta(x)}}\diff x.
	\end{align*}
	As before, by using the anisotropic Hardy's inequality, we conclude that
	\begin{align*}
		\varrho_{p(\cdot)}\l(\nabla \tilde{u}_n\r) \leq c_5 \l\|\tilde{u}_n\r\|\quad\text{for all }n\in\N\text{ and for some }c_5>0.
	\end{align*}
	Therefore, $\{\tilde{u}_n\}_{n\in\N}\subseteq\W$ is bounded.
	
	As in the proof of Proposition \ref{proposition_5} we have that  $\{\tilde{u}_n\}_{n\in\N}\subseteq C^1_0(\close)$ is relatively compact and so we may assume that
	\begin{align}\label{23}
		\tilde{u}_n\to \tilde{u}\quad\text{in }C^1_0(\close).
	\end{align}

	Moreover, if $\underline{u}\in\interior$ is the unique positive solution of
	\begin{align*}
		-\Delta_{p(\cdot)} u - \Delta_{q(\cdot)} u = 1\quad \text{in }\Omega, \quad u\big|_{\partial\Omega}=0,
	\end{align*}
	then by the weak comparison principle, we have $\underline{u}\leq \tilde{u}_n$ for all $n\in\N$. Hence, $\underline{u}\leq \tilde{u}$ and so $\tilde{u} \in\interior$. Furthermore, using \eqref{23} as $n\to \infty$ in the corresponding equation for $\tilde{u}_n$, we obtain
	\begin{align*}
		\l\lan A_{p(\cdot)} \l(\tilde{u}\r),h\r\ran +\l\lan A_{q(\cdot)} \l(\tilde{u}\r),h\r\ran = \into \l[1+\overline{u}_1^{-\eta(x)}\r]h\diff x
	\end{align*}
	for all $h\in\W$. Thus, $\tilde{u}\in\interior$ is a positive solution of \ref{problem_aux2}.
	
	On account of Proposition \ref{proposition_2} this positive solution is unique. Moreover we have
	\begin{align*}
		0&\leq \l\lan A_{p(\cdot)} \l(\overline{u}_1\r)-A_{p(\cdot)} \l(\tilde{u}\r),\l(\overline{u}_1-\tilde{u}\r)^+\r\ran + \l\lan A_{q(\cdot)} \l(\overline{u}_1\r)-A_{q(\cdot)} \l(\tilde{u}\r),\l(\overline{u}_1-\tilde{u}\r)^+\r\ran\\
		&=\into \l[\overline{u}_1^{-\eta(x)}-\l(1+\overline{u}_1^{-\eta(x)}\r)\r]\l(\overline{u}_1-\tilde{u}\r)^+\diff x\leq 0.
	\end{align*}
	This shows that $\overline{u}_1 \leq \tilde{u}$.
\end{proof}

We are going to apply $\overline{u}_\lambda, \tilde{u}\in\interior$ in order to show the nonemptiness of $\mathcal{L}$.

\begin{proposition}\label{proposition_7}
	If hypotheses H$_0$ and H$_1$ hold, then $\mathcal{L}\neq \emptyset$ and $\mathcal{S}_\lambda\subseteq \interior$ for every $\lambda \in \mathcal{L}$.
\end{proposition}

\begin{proof}
	Let $\lambda \in (0,1]$. Taking Propositions \ref{proposition_5} and \ref{proposition_6} into account, we define the Carath\'eodory function $\hat{g}_\lambda\colon\Omega\times\R\to\R$ by
	\begin{align}\label{24}
		\begin{split}
			\hat{g}_\lambda(x,s)=
			\begin{cases}
				\lambda \l[\overline{u}_\lambda^{-\eta(x)}+f\l(x,\overline{u}_\lambda(x)\r)\r] & \text{if }s<\overline{u}_\lambda(x),\\[1ex]
				\lambda \l[s^{-\eta(x)}+f\l(x,s\r)\r] & \text{if }\overline{u}_\lambda(x)\leq s \leq \tilde{u}(x),\\[1ex]
				\lambda \l[\tilde{u}^{-\eta(x)}+f\l(x,\tilde{u}(x)\r)\r] & \text{if }\tilde{u}(x)<s.
			\end{cases}
		\end{split}
	\end{align}
	
	We consider the following Dirichlet problem
	\begin{align}\label{25}
		-\Delta_{p(\cdot)} u - \Delta_{q(\cdot)} u = \hat{g}_\lambda(x,u)\quad \text{in }\Omega, \quad u\big|_{\partial\Omega}=0.
	\end{align}
	By using the direct method of the calculus of variations, we will produce a solution for problem \eqref{25} when $\lambda \in (0,1]$ is small enough. So, let $\hat{G}_\lambda(x,s)=\int^s_0\hat{g}_\lambda(x,t)\diff t$ and consider the $C^1$-functional $\hat{\ph}_\lambda\colon\W\to\R$ defined by
	\begin{align*}
		\hat{\ph}_\lambda(u)=\into \frac{1}{p(x)}|\nabla u|^{p(x)}\diff x+\into \frac{1}{q(x)}|\nabla u|^{q(x)}\diff x-\into \hat{G}_\lambda(x,u)\diff x
	\end{align*}
	for all $u\in\W$. From the definition of the truncation in \eqref{24} it is easy to see that
	\begin{align*}
		\hat{\ph}_\lambda(u)\geq \frac{1}{p_+} \l[\varrho_{p(\cdot)}(\nabla u)+\varrho_{q(\cdot)}(\nabla u)\r]-c_5
	\end{align*}
	for some $c_5>0$. Hence, $\hat{\ph}_\lambda\colon\W\to\R$ is coercive. Further $\hat{\ph}_\lambda\colon\W\to\R$ is sequentially weakly lower semicontinuous. Hence, there exists $u_\lambda\in\W$ such that
	\begin{align}\label{26}
		\hat{\ph}_\lambda(u_\lambda)=\min \l[\hat{\ph}_\lambda(u)\,:\,u\in \W\r].
	\end{align}
	
	Since $\tilde{u}\in\interior$, on account of hypothesis H$_1$(i) we can find $\lambda \in (0,1]$ small enough such that
	\begin{align}\label{27}
		\lambda f\l(x,\tilde{u}\r) \leq 1\quad\text{for a.\,a.\,}x\in\Omega.
	\end{align}

	From \eqref{26} we have $\hat{\ph}_\lambda'(u_\lambda)=0$, that is,
	\begin{align}\label{28}
		\l\lan A_{p(\cdot)}\l(u_\lambda\r),h\r\ran+\l\lan A_{q(\cdot)}\l(u_\lambda\r),h\r\ran=\into \hat{g}_\lambda(x,u_\lambda)h\diff x
	\end{align}
	for all $h\in\W$. First, we take $h=\l(\overline{u}_\lambda-u_\lambda\r)^+\in\W$ in \eqref{28}. Then, applying \eqref{24}, H$_1$(i) and Proposition \ref{5}, we obtain
	\begin{align*}
		&\l\lan A_{p(\cdot)}\l(u_\lambda\r),\l(\overline{u}_\lambda-u_\lambda\r)^+\r\ran+\l\lan A_{q(\cdot)}\l(u_\lambda\r),\l(\overline{u}_\lambda-u_\lambda\r)^+\r\ran\\
		& =\into \lambda \l[\overline{u}_\lambda^{-\eta(x)}+f\l(x,\overline{u}_\lambda\r)\r]\l(\overline{u}_\lambda-u_\lambda\r)^+\diff x\\
		& \geq\into \lambda \overline{u}_\lambda^{-\eta(x)}\l(\overline{u}_\lambda-u_\lambda\r)^+\diff x\\
		& = \l\lan A_{p(\cdot)}\l(\overline{u}_\lambda\r),\l(\overline{u}_\lambda-u_\lambda\r)^+\r\ran+\l\lan A_{q(\cdot)}\l(\overline{u}_\lambda\r),\l(\overline{u}_\lambda-u_\lambda\r)^+\r\ran.
	\end{align*}
	On account of Proposition \ref{proposition_2} we conclude that $\overline{u}_\lambda\leq u_\lambda$. Next, we choose $h=\l(u_\lambda-\tilde{u}\r)^+\in\W$ in \eqref{28}. Then, using \eqref{24}, \eqref{27} and Proposition \ref{proposition_6}, one has
	\begin{align*}
		&\l\lan A_{p(\cdot)}\l(u_\lambda\r),\l(u_\lambda-\tilde{u}\r)^+\r\ran+\l\lan A_{q(\cdot)}\l(u_\lambda\r),\l(u_\lambda-\tilde{u}\r)^+\r\ran\\
		& =\into \lambda \l[\tilde{u}^{-\eta(x)}+f\l(x,\tilde{u}\r)\r]\l(u_\lambda-\tilde{u}\r)^+\diff x\\
		& \leq \into \l[\tilde{u}^{-\eta(x)}+1\r]\l(u_\lambda-\tilde{u}\r)^+\diff x\\
		& = \l\lan A_{p(\cdot)}\l(\tilde{u}\r),\l(u_\lambda-\tilde{u}\r)^+\r\ran+\l\lan A_{q(\cdot)}\l(\tilde{u}\r),\l(u_\lambda-\tilde{u}\r)^+\r\ran.
	\end{align*}
	As before, from Proposition \ref{proposition_2} we see that $u_\lambda \leq \tilde{u}$. 
	
	In summary we have shown that $u_\lambda \in [\overline{u}_\lambda,\tilde{u}]$ for all $\lambda \in (0,1]$ small enough. From \eqref{24} and \eqref{28} we see that $u_\lambda$ is a solution of our original problem \eqref{problem}, that is, $u_\lambda \in\mathcal{S}_\lambda$. This proves the nonemptiness of $\mathcal{L}$.
	
	Let us now prove the second assertion of the proposition. To this end, let $u \in \mathcal{S}_\lambda$. Since $f \geq 0$ by hypothesis H$_1$(i), we have that $\overline{u}_\lambda\leq u$ and because $\overline{u}\in\interior$, there exists $c_6>0$ such that $c_6\hat{d} \leq u$, see Papageorgiou-R\u{a}dulescu-Repov\v{s} \cite[p.\,274]{13-Papageorgiou-Radulescu-Repovs-2019}. This fact, hypothesis H$_1$(i) and Theorem B1 of Saoudi-Ghanmi \cite{20-Saoudi-Ghanmi-2017} (see also Giacomoni-Schindler-Tak\'{a}\v{c} \cite{8-Giacomoni-Schindler-Takac-2007}), we have that $u\in\interior$. Therefore, $\mathcal{S}_\lambda\subseteq \interior$ for all $\lambda\in\mathcal{L}$.
\end{proof}

The next proposition shows that $\mathcal{L}$ is connected, that is, $\mathcal{L}$ is an interval.

\begin{proposition}
	If hypotheses H$_0$ and H$_1$ hold, $\lambda\in\mathcal{L}$ and $\mu \in (0,\lambda)$, then $\mu \in \mathcal{L}$.
\end{proposition}

\begin{proof}
	Since $\lambda\in\mathcal{L}$, there exists $u \in \mathcal{S}_\lambda\subseteq \interior$, see Proposition \ref{proposition_7}. Moreover, from Proposition \ref{proposition_5} we know that
	\begin{align}\label{29}
		\overline{u}_\mu \leq \overline{u}_\lambda \leq u.
	\end{align}
	
	Based on \eqref{29} we introduce the Carath\'eodory function $g_\mu\colon \Omega \times \R \to\R$ defined by
	\begin{align}\label{30}
		g_\mu(x,s)=
		\begin{cases}
			\mu \l[\overline{u}_\mu(x)^{-\eta(x)}+f\l(x,\overline{u}_\mu(x)\r)\r]&\text{if } s<\overline{u}_\mu(x),\\
			\mu\l[s^{-\eta(x)}+f\l(x,s\r)\r]&\text{if } \overline{u}_\mu(x)\leq s \leq u(x),\\
			\mu\l[u(x)^{-\eta(x)}+f\l(x,u(x)\r)\r]&\text{if } u(x) <s.
		\end{cases}
	\end{align}
	We set $G_\mu(x,s)=\int^s_0g_\mu(x,t)\diff t$ and consider the $C^1$-functional $\ph_\mu\colon \Wpzero{p(\cdot)}\to\R$ defined by
	\begin{align*}
		\ph_\mu(u)
		=\into \frac{1}{p(x)}|\nabla u|^{p(x)}\diff x+\into \frac{1}{q(x)}|\nabla u|^{q(x)}\diff x-\into G_\mu(x,u)\diff x
	\end{align*}
	for all $u \in \Wpzero{p(\cdot)}$. It is clear that $\ph_\mu$ is coercive because of \eqref{30} and it is sequentially weakly lower semicontinuous. So, there exists $u_\mu \in \Wpzero{p(\cdot)}$ such that
	\begin{align*}
		\ph_\mu(u_\mu)=\min\l[\ph_\mu(u)\,:\, u \in \Wpzero{p(\cdot)}\r].
	\end{align*}
	This implies, in particular, that $\ph'_\mu(u_\mu)=0$. Hence
	\begin{align}\label{31}
		\l \lan A_{p(\cdot)}\l(u_\mu\r),h\r\ran+\l \lan A_{q(\cdot)}\l(u_\mu\r),h\r\ran=\into g_\mu \l(x,u_\mu\r)h\diff x
	\end{align}
	for all $h \in \Wpzero{p(\cdot)}$. We first choose $h=\l(\overline{u}_\mu-u_\mu\r)^+\in\Wpzero{p(\cdot)}$ in \eqref{31}. Applying \eqref{30}, hypothesis H$_1$(i) and Proposition \ref{proposition_5} yields
	\begin{align*}
		&\l \lan A_{p(\cdot)}\l(u_\mu\r),\l(\overline{u}_\mu-u_\mu\r)^+\r\ran+\l \lan A_{q(\cdot)}\l(u_\mu\r),\l(\overline{u}_\mu-u_\mu\r)^+\r\ran\\
		&=\into \mu \l[\overline{u}_\mu^{-\eta(x)}+f\l(x,\overline{u}_\mu\r)  \r]\l(\overline{u}_\mu-u_\mu\r)^+\diff x\\
		&\geq \into \mu \overline{u}_\mu^{-\eta(x)}\l(\overline{u}_\mu-u_\mu\r)^+\diff x\\
		&=\l \lan A_{p(\cdot)}\l(\overline{u}_\mu\r),\l(\overline{u}_\mu-u_\mu\r)^+\r\ran+\l \lan A_{q(\cdot)}\l(\overline{u}_\mu\r),\l(\overline{u}_\mu-u_\mu\r)^+\r\ran.
	\end{align*}
	Proposition \ref{proposition_2} then implies that $\overline{u}_\lambda \leq u_\mu$. Now  we choose $h=\l(u_\mu-u\r)^+\in\Wpzero{p(\cdot)}$ in \eqref{31}. Then, from \eqref{30}, $\mu<\lambda$ and $u\in\mathcal{S}_\lambda$, we derive
	\begin{align*}
		&\l \lan A_{p(\cdot)}\l(u_\mu\r),\l(u_\mu-u\r)^+\r\ran+\l \lan A_{q(\cdot)}\l(u_\mu\r),\l(u_\mu-u\r)^+\r\ran\\
		&=\into\mu \l[u^{-\eta(x)}+f\l(x,u_\lambda\r)\r]\l(u_\mu-u\r)^+\diff x\\
		& \leq \into \lambda \l[u^{-\eta(x)}+ f\l(x,u\r)\r]\l(u_\mu-u\r)^+\diff x\\
		& =\l \lan A_{p(\cdot)}\l(u \r),\l(u_\mu-u\r)^+\r\ran+\l \lan A_{q(\cdot)}\l(u\r),\l(u_\mu-u\r)^+\r\ran.
	\end{align*}
	Thus, $u_\mu \leq u$. Therefore we have proved that
	\begin{align}\label{32}
		u_\mu \in \l[\overline{u}_\mu,u\r].
	\end{align}
	From \eqref{32}, \eqref{30} and \eqref{31} it follows that
	\begin{align*}
		u_\mu \in \mathcal{S}_\mu \subseteq \interior
	\end{align*}
	and so $\mu \in \mathcal{L}$.
\end{proof}

An immediate consequence of the proof above is the following corollary.

\begin{corollary}\label{corollary_9}
	If hypotheses H$_0$ and H$_1$ hold and if $\lambda \in\mathcal{L}, u \in\mathcal{S}_\lambda \subseteq \interior$ and $0<\mu<\lambda$, then $\mu\in\mathcal{L}$ and there exists $u_\mu \in \mathcal{S}_\mu \subseteq \interior$ such that $u_\mu \leq u$.
\end{corollary}

We can improve the conclusion of this corollary.

\begin{proposition}\label{proposition_10}
	If hypotheses H$_0$ and H$_1$ hold and $\lambda \in\mathcal{L}, u \in\mathcal{S}_\lambda \subseteq \interior$ and $0<\mu<\lambda$, then $\mu\in\mathcal{L}$ and there exists $u_\mu \in \mathcal{S}_\mu \subseteq \interior$ such that
	\begin{align*}
		u-u_\mu \in \interior.
	\end{align*}
\end{proposition}

\begin{proof}
	From Corollary \ref{corollary_9} we already know that $\mu \in \mathcal{L}$ and that we can find $u_\mu \in \mathcal{S}_\mu \subseteq \interior$ such that
	\begin{align}\label{33}
		\overline{u}_\mu\leq u_\mu \leq u.
	\end{align}
	
	Now, let $\rho=\|u\|_\infty$ and let $\hat{\xi}_\rho>0$ be as given in hypothesis H$_1$(v). Since $\mu<\lambda$, $u_\mu\in\mathcal{S}_\mu$ and due to \eqref{33}, hypothesis H$_1$(v) and $f \geq 0$, we have
	\begin{align}\label{34}
		\begin{split}
			&-\Delta_{p(\cdot)}u_\mu-\Delta_{q(\cdot)}u_\mu+\lambda\hat{\xi}_\rho u_\mu^{p(x)-1}-\lambda u_\mu^{-\eta(x)}\\
			& < -\Delta_{p(\cdot)}u_\mu-\Delta_{q(\cdot)}u_\mu+\lambda\hat{\xi}_\rho u_\mu^{p(x)-1}-\mu u_\mu^{-\eta(x)}\\
			&=\mu f \l(x,u_\mu\r)+\lambda \hat{\xi}_\rho u_\mu^{p(x)-1}\\
			& \leq \lambda \l[f \l(x,u_\mu\r)+\hat{\xi}_\rho u_\mu^{p(x)-1}\r]-(\lambda-\mu)f(x,u_\mu)\\
			& \leq \lambda  \l[ f \l(x,u\r)+\hat{\xi}_\rho u^{p(x)-1}\r]\\
			&=-\Delta_{p(\cdot)}u-\Delta_{q(\cdot)}u+\lambda\hat{\xi}_\rho u^{p(x)-1}-\lambda u^{-\eta(x)}.
		\end{split}
	\end{align}
	
	Since $u_\mu \in \interior$, using hypothesis H$_1$(iv), we see that
	\begin{align*}
		0 \preceq [\lambda -\mu]f(\cdot,u_\mu(\cdot)).
	\end{align*}
	Then, from \eqref{34} and Proposition \ref{proposition_3}, we conclude that
	\begin{align*}
		u-u_\mu \in \interior.
	\end{align*} 
\end{proof}

\begin{remark}
	In the same way as in the proof of Proposition \ref{proposition_10}, we can also show that
	\begin{align}\label{35}
		u_\mu-\overline{u}_\mu \in \interior.
	\end{align} 
\end{remark}

Let $\lambda^*=\sup \mathcal{L}$. The next proposition shows that $\lambda^*$ is finite.

\begin{proposition}
	If hypotheses H$_0$ and H$_1$ hold, then $\lambda^*<+\infty$.
\end{proposition}

\begin{proof}
	From Hypotheses H$_1$(i)--(iv) we see that there exists $\hat{\lambda}>0$ large enough such that
	\begin{align}\label{36}
		\hat{\lambda}f(x,s)\geq s^{p(x)-1} \quad\text{for a.\,a.\,}x\in \Omega \text{ and for all }s\geq 0.
	\end{align}
	Let $\lambda>\hat{\lambda}$ and suppose that $\lambda \in \mathcal{L}$. Then we can find $u\in \mathcal{S}_\lambda\subseteq \interior$. Let $\Omega_0\subseteq \Omega$ be an open subset with $C^2$-boundary such that $\overline{\Omega}_0\subseteq \Omega$ and $u$ is not constant on $\close_0$. We define $m_0=\min_{x\in \overline{\Omega}_0}u(x)$. Since $u \in \interior$ it is clear that $m_0>0$. For $\delta \in (0,\|u\|_\infty-m_0)$ we set $m_0^\delta=m_0+\delta$. Further, for $\rho=\|u\|_\infty$ let $\hat{\xi}_\rho>0$ be as given by hypothesis H$_1$(v). First, for $\delta$ small enough, we observe that 
	\begin{align}\label{37}
		\frac{1}{m_0^{\eta(x)}}-\frac{1}{\l(m_0+\delta\r)^{\eta(x)}}
		= \frac{(m_0+\delta)^{\eta(x)}-m_0^{\eta(x)}}{\l[m_0(m_0+\delta)\r]^{\eta(x)}}
		\leq \l(\frac{\delta}{m_0^2}\r)^{\eta(x)} \leq \l(\frac{\delta}{m_0^2}\r)^{\eta_-}
	\end{align}
	for all $x \in \close$. Then, applying \eqref{37}, \eqref{36}, hypotheses H$_1$(iv), (v), $u\in\mathcal{S}_\lambda$ and $\delta>0$ small enough, we have
	\begin{align}\label{38}
		\begin{split}
			& -\Delta_{p(\cdot)}m_0^\delta-\Delta_{q(\cdot)}m_0^\delta+\l[\lambda \hat{\xi}_\rho+1\r] \l(m_0^\delta\r)^{p(x)}-\lambda \l(m_0^\delta\r)^{-\eta(x)}\\
			& \leq \l[\lambda \hat{\xi}_\rho+1\r] m_0^{p(x)-1}+\chi(\delta) \quad\text{with }\chi(\delta)\to 0^+ \text{ as }\delta\to 0^+,\\
			& \leq \hat{\lambda} f(x,m_0)+ \lambda \hat{\xi}_\rho m_0^{p(x)-1}+\chi(\delta)\\
			& = \lambda \l[f(x,m_0)+ \hat{\xi}_\rho m_0^{p(x)-1}\r]-\l(\lambda-\hat{\lambda}\r)f(x,m_0)+\chi(\delta)\\
			& \leq \lambda \l[f(x,m_0)+ \hat{\xi}_\rho m_0^{p(x)-1}\r]\\
			& \leq \lambda \l[f(x,u)+ \hat{\xi}_\rho u^{p(x)-1}\r]\\
			& =-\Delta_{p(\cdot)}u-\Delta_{q(\cdot)}u+\lambda\hat{\xi}_\rho u^{p(x)-1}-\lambda u^{-\eta(x)} \quad \text{in }\Omega_0.
		\end{split}	
	\end{align}
	
	For $\delta>0$ small enough, because of hypothesis H$_1$(iv), we know that
	\begin{align*}
		0<\tilde{\eta}_0 \leq \l[\lambda -\hat{\lambda}\r]f(x,m_0)-\chi(\delta)
.	\end{align*}
	Then, from \eqref{38} and Proposition \ref{proposition_4}, we infer that
	\begin{align*}
		0<u(x)-m_0^\delta\quad \text{for all }x\in \Omega \text{ and for all small }\delta>0.
	\end{align*}
	This is a contradiction to the definition of $m_0>0$. Therefore, $\lambda \not\in\mathcal{L}$ and so $\lambda^*\leq \hat{\lambda}<\infty$. 
\end{proof}

We have just proved that $(0,\lambda^*)\subseteq \mathcal{L}\subseteq (0,\lambda^*]$. Next we show that our original problem \eqref{problem} has at least two positive smooth solution for $\lambda \in (0,\lambda^*)$.

\begin{proposition}\label{proposition_12}
	If hypotheses H$_0$ and H$_1$ hold and if $\lambda \in (0,\lambda^*)$, then problem \eqref{problem} has at least two positive solutions
	\begin{align*}
		u_0, \hat{u} \in \interior \text{ with }u_0 \neq \hat{u}.
	\end{align*}
\end{proposition}

\begin{proof}
	Let $\vartheta \in (\lambda,\lambda^*)\subseteq \mathcal{L}$ and let $u_\vartheta \in \mathcal{S}_\vartheta\subseteq \interior$. From Proposition \ref{proposition_10} and \eqref{35} we know there exists $u_0 \in \mathcal{S}_\lambda \subseteq \interior$ such that
	\begin{align}\label{39}
		u_0 \in \sideset{}{_{C^1_0(\close)}} \ints [\overline{u}_\lambda,u_\vartheta].
	\end{align}
	
	We introduce the Carath\'eodory function $k_\lambda\colon \Omega\times\R\to\R$ defined by
	\begin{align}\label{40}
		k_\lambda(x,s)=
		\begin{cases}
			\lambda \l[\overline{u}_\lambda(x)^{-\eta(x)}+ f\l(x,\overline{u}_\lambda(x)\r)\r] &\text{if }s \leq \overline{u}_\lambda(x),\\
			\lambda \l[s^{-\eta(x)}+f\l(x,s\r)\r] &\text{if }\overline{u}_\lambda(x)<s.
		\end{cases}
	\end{align}

	We set $K_\lambda(x,s)=\int^s_0 k_\lambda(x,t)\diff t$  and consider the $C^1$-functional $\sigma_{\lambda}\colon \Wpzero{p(\cdot)}\to \R$ defined by
	\begin{align*}
		\sigma_\lambda(u)
		&=\into \frac{1}{p(x)}|\nabla u|^{p(x)}\diff x +\into \frac{1}{q(x)}|\nabla u|^{q(x)}\diff x-\into K_\lambda (x,u)\diff x
	\end{align*}
	for all $u \in \Wpzero{p(\cdot)}$.

	Using \eqref{40} we can easily show that
	\begin{align}\label{41}
		K_{\sigma_\lambda}\subseteq [\overline{u}_\lambda) \cap \interior.
	\end{align}
	Hence we may assume that
	\begin{align}\label{42}
		K_{\sigma_\lambda}\cap [\overline{u}_\lambda,u_\vartheta]=\{u_0\},
	\end{align}
	otherwise we already have a second positive smooth solution of \eqref{problem} and so we are done, see \eqref{41} and \eqref{40}.
	
	We truncate $k_\lambda(x,\cdot)$ at $u_\vartheta(x)$. This is done by the Carath\'eodory function $\hat{k}_\lambda\colon\Omega\times\R\to\R$ defined by
	\begin{align}\label{43}
		\hat{k}_\lambda(x,s)=
		\begin{cases} 
			k_\lambda(x,s) & \text{if } s \leq u_\vartheta(x),\\
			k_\lambda\l(x,u_\vartheta(x)\r) & \text{if } u_\vartheta(x)<s.
		\end{cases} 
	\end{align}
	We set $\hat{K}_\lambda(x,s)=\int^s_0 \hat{k}_\lambda(x,t)\diff t$ and consider the $C^1$-functional $\hat{\sigma}_\lambda\colon \Wpzero{p(\cdot)}\to \R$ defined by
	\begin{align*}
		\hat{\sigma}_\lambda(u)&=\into \frac{1}{p(x)}|\nabla u|^{p(x)}\diff x +\into \frac{1}{q(x)}|\nabla u|^{q(x)}\diff x-\into \hat{K}_\lambda (x,u)\diff x
	\end{align*}
	for all $u \in \Wpzero{p(\cdot)}$.
	
	Looking at \eqref{40} and \eqref{43} we see that
	\begin{align}\label{44}
		\hat{\sigma}_\lambda \big|_{[0,u_\vartheta]}=\sigma_\lambda \big|_{[0,u_\vartheta]}
		\quad\text{and}\quad
		\hat{\sigma}'_\lambda \big|_{[0,u_\vartheta]}=\sigma'_\lambda \big|_{[0,u_\vartheta]}.
	\end{align}
	Further, from \eqref{43} it is clear that
	\begin{align}\label{45}
		K_{\hat{\sigma}_\lambda}\subseteq [\overline{u}_\lambda,u_\vartheta]\cap \interior.
	\end{align}
	
	From the definition of the truncations in \eqref{40} and \eqref{43} we know that $\hat{\sigma}_\lambda$ is coercive and it is also sequentially weakly lower semicontinuous. Thus, we can find $\tilde{u}_0\in\Wpzero{p(\cdot)}$ such that
	\begin{align*}
		\hat{\sigma}_\lambda\l(\tilde{u}_0\r)=\min \l[\hat{\sigma}_\lambda(u)\,:\,u\in\Wpzero{p(\cdot)}\r].
	\end{align*}
	Taking \eqref{45}, \eqref{44}, \eqref{42} into account we conclude that $\tilde{u}_0=u_0$. Then, on account of \eqref{39} and \eqref{44}, $u_0 \in \interior$ is a local $C^1_0(\close)$-minimizer of $\sigma_\lambda$. The results of Tan-Fang \cite{21-Tan-Fang-2013} imply that 
	\begin{align}\label{46}
		u_0\in\interior \text{is a $\W$-minimizer of $\sigma_\lambda$.}
	\end{align}
	
	From \eqref{41} it is clear that we may assume that $K_{\sigma_\lambda}$ is finite otherwise we would have a sequence of distinct positive solutions of \eqref{problem} and so we would have done. The finiteness of $K_{\sigma_\lambda}$ along with \eqref{46} and Theorem 5.7.6 of Papageorgiou-R\u{a}dulescu-Repov\v{s} \cite[p.\,449]{13-Papageorgiou-Radulescu-Repovs-2019} imply that we can find $\hat{\rho} \in (0,1)$ small enough such that
	\begin{align}\label{47}
		\sigma_\lambda(u_0)<\inf \l[\sigma_\lambda(u)\,:\, \|u-u_0\|=\hat{\rho} \r]=m_\lambda.
	\end{align}
	
	Reasoning as in the proof of Proposition 4.1 of Gasi\'nski-Papageorgiou \cite{7-Papageorgiou-Gasinski-2011} we can show that
	\begin{align}\label{48}
		\sigma_\lambda \text{ satisfies the $C$-condition}.
	\end{align}
	
	Moreover, if $u \in \interior$, then on account of hypothesis H$_1$(ii) and \eqref{40}, we have 
	\begin{align}\label{49}
		\sigma_\lambda(tu)\to -\infty \quad\text{as }t\to +\infty.
	\end{align}

	Then, \eqref{47}, \eqref{48} and \eqref{49} permit us the use of the mountain pass theorem. Hence, there exists $\hat{u}\in\Wpzero{p(\cdot)}$ such that
	\begin{align*}
		\hat{u} \in K_{\sigma_\lambda} \subseteq [\overline{u}_\lambda)\cap \interior,
	\end{align*}
	see \eqref{41}, and 
	\begin{align*}
		m_\lambda\leq \sigma_\lambda \l(\hat{u}\r),
	\end{align*}
	see \eqref{47}. Taking \eqref{40} and \eqref{47} into account we conclude that $\hat{u} \in \interior$ is  a solution of \eqref{problem} for $\lambda \in \l(0,\lambda^*\r)$ with $\hat{u}\neq u_0$.
\end{proof}

Next we will check the admissibility of the critical parameter $\lambda^*>0$.

\begin{proposition}
	If hypotheses H$_0$ and H$_1$ hold, then $\lambda^* \in \mathcal{L}$, that is, $\mathcal{L}=(0,\lambda^*]$.
\end{proposition}

\begin{proof}
	Let $\{\lambda_n\}_{n\in\N} \subseteq (0,\lambda^*)\subseteq \mathcal{L}$ be such that $\lambda_n\nearrow \lambda^*$ as $n \to \infty$. Let $\overline{u}_1=\overline{u}_{\lambda_1}\in\interior$ be the unique solution of \eqref{problem_aux} for $\lambda=\lambda_1$ obtained in Proposition \ref{proposition_5}. By hypothesis H$_1$(i) we know that $f \geq 0$. Then from \eqref{40} we get that $\sigma_{\lambda_1}(\overline{u}_1)\leq 0$. Hence,
	\begin{align}\label{50}
		\sigma_{\lambda_n}\l(\overline{u}_1\r) \leq 0 \quad \text{for all }n\in\N,
	\end{align}
	since $\lambda_1\leq \lambda_n$ for all $n \in \N$.
	
	From the proof of Proposition \ref{proposition_12} we know there exists $u_n \in \mathcal{S}_{\lambda_n} \subseteq \interior$ such that $\overline{u}_1 \leq u_n$ and 
	\begin{align}\label{51}
		\sigma_{\lambda_n}(u_n) \leq \sigma_{\lambda_n}(\overline{u}_1)\leq 0 \quad\text{for all }n\in\N,
	\end{align}
	see \eqref{50}. Since $u_n \in \mathcal{S}_{\lambda_n}$ it holds
	\begin{align}\label{52}
		\sigma'_{\lambda_n}(u_n)=0\quad\text{for all }n\in\N.
	\end{align}

	From \eqref{51}, \eqref{52} and Proposition 4.1 of Gasi\'nski-Papageorgiou \cite{7-Papageorgiou-Gasinski-2011} we can conclude that $\{u_n\}_{n\in\N}\subseteq \W$ is bounded. So, we may assume that
	\begin{align}\label{53}
		u_n\weak u_* \quad\text{in }\W\quad\text{and}\quad u_n\to u_* \quad\text{in }\Lp{r(\cdot)}.
	\end{align}
	
	From \eqref{52} we have
	\begin{align}\label{54}
		\l \lan A_{p(\cdot)}\l(u_n\r),h\r\ran+\l \lan A_{q(\cdot)}\l(u_n\r),h\r\ran=\into k_\lambda \l(x,u_n\r)h\diff x
	\end{align}
	for all $h \in \Wpzero{p(\cdot)}$ and for all $n \in \N$.
	
	We take $h=u_n-u_*\in\W$ as test function \eqref{54}. Applying \eqref{53} and hypothesis H$_1$(i) gives
	\begin{align*}
		\lim_{n\to\infty} \l[\l\lan A_{p(\cdot)}(u_n),u_n-u_*\r\ran+\l\lan A_{q(\cdot)}(u_n),u_n-u_*\r\ran\r]=0.
	\end{align*}	
	Since $A_{q(\cdot)}$ is monotone, see Proposition \ref{proposition_2}, we obtain
	\begin{align*}
		\limsup_{n\to\infty} \l[\l\lan A_{p(\cdot)}(u_n),u_n-u_*\r\ran+\l\lan A_{q(\cdot)}(u_*),u_n-u_*\r\ran\r]\leq 0.
	\end{align*}
	Then, by using \eqref{53}, it follows
	\begin{align*}
		\limsup_{n\to\infty} \l\lan A_{p(\cdot)}(u_n),u_n-u_*\r\ran\leq 0.
	\end{align*}
	From this and Proposition \ref{proposition_2} we conclude that 
	\begin{align}\label{55}
		u_n\to u_* \quad\text{in }\W\quad\text{and}\quad \overline{u}_1 \leq u_*.
	\end{align}
	If we now pass to the limit in \eqref{54} as $n\to \infty$, then, by applying \eqref{55}, we see that $u_* \in \mathcal{S}_{\lambda^*}$ and so $\lambda^*\in\mathcal{L}$, that is, $\mathcal{L}=(0,\lambda^*]$.
\end{proof}

In summary, we can state the following bifurcation-type result concerning problem \eqref{problem}.

\begin{theorem}
	If hypotheses H$_0$ and H$_1$ hold, then there exists $\lambda^*>0$ such that
	\begin{enumerate}
		\item[(a)]
			for every $\lambda\in (0,\lambda^*)$, problem \eqref{problem} has at least two positive solutions
		\begin{align*}
		u_0, \hat{u} \in \interior, \quad u_0\neq \hat{u};
		\end{align*}
		\item[(b)]
		for $\lambda=\lambda^*$, problem \eqref{problem} has at least one positive solution
		\begin{align*}
		u_*\in\interior;
		\end{align*}
		\item[(c)]
		for every $\lambda>\lambda^*$, problem \eqref{problem} has no positive solutions.
	\end{enumerate}
\end{theorem}

\section{Minimal positive solutions}

In this section we are going to show that for every admissible parameter $\lambda\in\mathcal{L}=(0,\lambda^*]$, problem \eqref{problem} has a smallest positive solution (so-called minimal positive solution) $\tilde{u}_\lambda\in\mathcal{S}_\lambda\subseteq \interior$, that is, $\tilde{u}_\lambda \leq u$ for all $u \in \mathcal{S}_\lambda$. Moreover, we determine the monotonicity and continuity properties of the minimal solution map $\mathcal{L}\ni\lambda \to \tilde{u}_\lambda\in \interior$.

\begin{proposition}
	If hypotheses H$_0$ and H$_1$ hold and if $\lambda \in\mathcal{L}\in (0,\lambda^*]$, then problem \eqref{problem} has a smallest positive solution $\tilde{u}_\lambda\in\interior$.
\end{proposition}

\begin{proof}
	As in the proof of Proposition 18 in Papageorgiou-R\u{a}dulescu-Repov\v{s} \cite{16-Papageorgiou-Radulescu-Repovs-2017}, we show that the set $\mathcal{S}_\lambda$ is downward directed, that is, if $u,v\in\mathcal{S}_\lambda$, then there exists $y\in\mathcal{S}_\lambda$ such that $y\leq u$ and $y \leq v$. Invoking Lemma 3.10 of Hu-Papageorgiou \cite[p.\,178]{11-Hu-Papageorgiou-1997}, we can find a decreasing sequence $\{u_n\}_{n\in\N}\subseteq \mathcal{S}_\lambda$ such that
	\begin{align}\label{56}
		\inf \mathcal{S}_\lambda=\inf_{n\in\N}u_n \quad\text{and}\quad \overline{u}_\lambda\leq u_n\leq u_1 \quad\text{for all }n\in\N.
	\end{align}
	
	From \eqref{56} it follows that the sequence $\{u_n\}_{n\in\N}\subseteq\W$ is bounded. So we may assume that
	\begin{align}\label{57}
		u_n\weak \tilde{u}_\lambda \quad\text{in }\W\quad\text{and}\quad u_n\to\tilde{u}_\lambda \quad\text{in }\LP.
	\end{align}
	Since $u_n\in\mathcal{S}_\lambda$, we have
	\begin{align}\label{58}
		\l \lan A_{p(\cdot)}\l(u_n\r),h\r\ran+\l \lan A_{q(\cdot)}\l(u_n\r),h\r\ran
		=\into\lambda \l[u_n^{-\eta(x)}+f(x,u_n)\r]h\diff x
	\end{align}
	for all $h \in \Wpzero{p(\cdot)}$ and for all $n \in \N$. Note that 
	\begin{align*}
		0 \leq u_n^{-\eta(\cdot)} \leq \overline{u}_\lambda^{-\eta(\cdot)} \in \Lp{1},
	\end{align*}
	see Lazer-McKenna \cite{12-Lazer-McKenna-1991}.
	
	We choose $h=u_n-\tilde{u}_\lambda\in\W\cap \Linf$ in \eqref{58}, pass to the limit as $n\to\infty$ and apply \eqref{57}. This yields
	\begin{align*}
		\limsup_{n\to\infty} \l\lan A_{p(\cdot)}(u_n),u_n-\tilde{u}_\lambda\r\ran\leq 0.
	\end{align*}
	Then, from Proposition \ref{proposition_2}, it follows
	\begin{align}\label{59}
		u_n\to \tilde{u}_\lambda \quad\text{in }\W\quad\text{and}\quad\overline{u}_\lambda \leq \tilde{u}_\lambda.
	\end{align}
	
	Passing to the limit in \eqref{58} as $n\to\infty$ and using \eqref{59}, we obtain
	\begin{align*}
		\tilde{u}_\lambda \in\mathcal{S}_\lambda \subseteq \interior\quad \text{and}\quad \tilde{u}_\lambda=\inf \mathcal{S}_\lambda.
	\end{align*}
\end{proof}

\begin{proposition}
	If hypotheses H$_0$ and H$_1$ hold, then the map $\lambda \to \tilde{u}_\lambda$ from $\mathcal{L}=(0,\lambda^*]$ into $C^1_0(\close)$ is
	\begin{enumerate}
		\item[(a)]
			strictly increasing, that is, $0<\lambda'<\lambda$ implies $\tilde{u}_\lambda-\tilde{u}_\lambda\in\interior$;
		\item[(b)]
			left continuous.
	\end{enumerate}
\end{proposition}

\begin{proof}
	(a) This is an immediate consequence of Proposition \ref{proposition_10}.
	
	(b) Let $\{\lambda_n\}_{n\in\N}\subseteq \mathcal{L}$ be a sequence such that $\lambda_n\to \lambda^-$. We have 
	\begin{align*}
		\overline{u}_{\lambda_1} \leq \tilde{u}_{\lambda_n} \leq \tilde{u}_\lambda\quad\text{for all }n\in\N.
	\end{align*}
	Hence, $\{\tilde{u}_{\lambda_n}\}_{n\in\N}\subseteq\W$ is bounded.
	
	Then, as before, see the proof of Proposition \ref{proposition_5}, via the anisotropic regularity theory, there exist $\alpha\in(0,1)$ and $c_7>0$ such that
	\begin{align}\label{60}
		\tilde{u}_{\lambda_n}\in C^{1,\alpha}_0(\close)\quad \text{and}\quad \l\|\tilde{u}_{\lambda_n}\r\|_{C^{1,\alpha}_0(\close)}\leq c_7\quad\text{for all }n\in\N.
	\end{align}
	Since $C^{1,\alpha}_0(\close)$ is compactly embedded into $C^{1}_0(\close)$, from \eqref{60} it follows that we have at least for a subsequence
	\begin{align}\label{61}
		\tilde{u}_{\lambda_n}\to \hat{u}_\lambda \quad\text{in }C^{1}_0(\close) \quad\text{and}\quad \hat{u}_\lambda \in\mathcal{S}_\lambda\subseteq\interior.
	\end{align}

	Suppose that $\hat{u}_\lambda\neq \tilde{u}_\lambda$. Then there exists $x\in\Omega$ such that $\tilde{u}_\lambda(x)<\hat{u}_\lambda(x)$. Then
	\begin{align*}
		\tilde{u}_\lambda(x)<\tilde{u}_{\lambda_n}(x)\quad\text{for all }n\in\N,
	\end{align*}
	see \eqref{61}. But this contradicts (a). Hence, $\hat{u}_\lambda= \tilde{u}_\lambda$ and by Urysohn's criterion for convergent sequences, we have $\tilde{u}_{\lambda_n}\to \hat{u}_\lambda$ in $C^{1}_0(\close)$ for the initial sequence. Therefore, $\lambda \to \tilde{u}_\lambda$ is left continuous from $\mathcal{L}=(0,\lambda^*]$ into $C^1_0(\close)$.
\end{proof}

\end{document}